\documentclass{amsart}
\usepackage{enumerate}
\usepackage{calc}
\usepackage[dvips]{graphicx}
\usepackage{color}
\usepackage{pstricks}
\usepackage{amsmath, amssymb}
\usepackage{pst-node}
\newtheorem{theo}{Theorem}[section]

\newtheorem{lem}[theo]{Lemma}

\newtheorem{prop}[theo]{Proposition}

\newtheorem{defn}[theo]{Definition}

\theoremstyle{definition}

\newtheorem*{notaetoile}{Notation}

\def\a{{\a}}
\def\a{{\mathfrak a}}

\def\C{\mathbb C}

\def\E{\mathbb E}

\def\N{\mathbb N}

\def\Q{\mathbb Q}
\def\R{\mathbb R}

\def\sh{\operatorname{sh}}

\def\Z{\mathbb{Z}}
\begin{document} 

\title[$\hat{\mathfrak{sl}_2}(\C)$ and a conditioned space-time Brownian motion]{The affine Lie algebra $\hat{\mathfrak{sl}_2}(\C)$ and a conditioned space-time Brownian motion}

\author{Manon Defosseux}
\address{Laboratoire de Math\'ematiques Appliqu\'ees \`a Paris 5, Universit\'e Paris 5, 45 rue des  Saints P\`eres, 75270 Paris Cedex 06.}
\email{manon.defosseux@parisdescartes.fr}

\begin{abstract} We construct a sequence of Markov processes on the set of dominant weights of the Affine Lie algebra $\hat{\mathfrak{sl}_2}(\C)$ which involves tensor product of irreducible highest weight modules of $\hat{\mathfrak{sl}_2}(\C)$ and show that it converges towards a Doob's space-time harmonic transformation of  a space-time Brownian motion.
\end{abstract}
 
\maketitle  
\section{Introduction}
In \cite{BBO}, Ph. Biane, Ph. Bougerol and N. O'Connell establish a wide extension of Pitman's theorem on Brownian motion and three dimensional Bessel process, in the framework of representation theory of semi-simple complex Lie algebras. In this framework the representation of the Bessel process by a functional of a standard Brownian motion $(B_t)_{t\ge 0}$ on $\R$,
$$(B_t-2\underset{0\le s\le t}\inf B_s, t\ge 0),$$ 
appears to be the continuous counterpart of a similar result which holds for a random walk on the set of integral weights of $\mathfrak{sl}_2(\C)$ and a path transformation connected with the Littelmann paths model for semi-simple complex Lie algebras (see for instance \cite{Littelmann} for a description of this model). 

In \cite{LLP}, C. Lecouvey, E. Lesigne and M. Peign\'e consider the case when $\mathfrak{g}$ is a Kac-Moody algebra and develop some aspects of \cite{BBO} in that framework. In particular, they focus on some Markov chains on the Weyl chamber of a Kac-Moody algebra, which are obtained in a similar way as in \cite{BBO}, except that the reference measure can't be the uniform measure when the dimension of the Kac-Moody algebra is infinite. Let us say briefly how the Markov chains are obtained for a Kac-Moody algebra $\mathfrak{g}$. As in the finite dimensional case, for  a dominant integral weight $\lambda$ of $\mathfrak{g}$ one defines the character of the irreducible highest-weight representation $V(\lambda)$ of $\mathfrak{g}$ with highest weight $\lambda$, as a formal series
$$\mbox{ch}_\lambda=\sum_{\mu}\mbox{dim} (V(\lambda)_\mu)e^{ \mu},$$
where $V(\lambda)_\mu$ is the weight space of $V(\lambda)$ corresponding to the weight $\mu$. Actually  for every $h$ in a subset of the Cartan subalgebra which doesn't depend on $\lambda$ the series 
$\sum_{\mu}\mbox{dim} (V(\lambda)_\mu)e^{\langle \mu,h\rangle}$ is absolutely convergente.  For  two dominant weights $\omega$ and $\lambda$, the following decomposition
$$\mbox{ch}_{\omega}\mbox{ch}_{\lambda}=\sum_{\beta\in P_+} m_{\lambda}(\beta)\mbox{ch}_{\beta},$$
where $m_{\lambda}(\beta)$ is the multiplicity of the module with highest weight $\beta$ in the decomposition of $V(\omega)\otimes V(\lambda)$, allows to define a transition probability  $q_\omega$ on the set of dominant weights, letting 
for $\beta$ and $\lambda$ two dominant weights of $\mathfrak{g}$, 
\begin{align*}
q_\omega(\lambda,\beta)=\frac{\mbox{ch}_\beta(h)}{\mbox{ch}_\lambda(h) \mbox{ch}_\omega(h)}m_{\lambda}(\beta),
\end{align*}
 where $h$ is chosen in the region of convergence of the characters. It is a natural question to ask if there exists a  sequence $(h_n)_{n\ge 0}$ of elements of $\mathfrak{h}$ such that the corresponding sequence of Markov chains  converges towards a continuous process and what the limit is. 

In this paper, we consider the case when $\mathfrak{g}$ is the Kac-Moody algebra of type $A^{(1)}_1$  and $\omega$ is  its fundamental weight $\Lambda_0$. There is no reason to think that the results are not true in a more general context but this case presents the advantage that explicit computations can easily be done. We show that the sequence of Markov chains, with a proper normalization, converges, for a particular sequence of $(h_n)_{n\ge 0}$, towards a Doob's space-time harmonic transformation of a space-time Brownian motion killed on the boundary of a time-dependent domain. This process is related to the heat equation 
$$\frac{1}{2}\Delta+\frac{\partial}{\partial t}=0,$$
in a time-dependent domain, with Dirichlet boundary conditions and the theta functions play a crucial role in the construction.  One can find an extensive literature devoted to the relationship between Brownian motion and the heat equation. One can see for instance \cite{Doob} for an introduction and \cite{Crank} for a review of various problems specifically related to time-dependent boundaries.

The paper is organized as follows. Basic definitions and notations related to representation theory of the affine Lie algebra $\hat{\mathfrak{sl}_2}(\C)$ are given  in section \ref{Basic notations and definitions}. We define in section \ref{Markovchains} random walks on the set of integral weights of $\hat{\mathfrak{sl}_2}(\C)$ and Markov chains on the set of its dominant integral weights, considering tensor products of irreducible highest weight representations of $\hat{\mathfrak{sl}_2}(\C)$. In section \ref{Brownian}, for any positive real numbers $x$ and $u$ such that $x<u$, we define a space-time Brownian motion $(t+u,B_t)_{t\ge 0}$ starting from $(u,x)$, conditioned to remain in the domain $$D=\{(r,z)\in \R\times \R: 0\le z\le r\}.$$ For this we  introduce a space-time harmonic function   remaining positive on $D$ which appears naturally considering  the limit of a sequence of characters of $\hat{\mathfrak{sl}_2}(\C)$. We prove in section \ref{convergence} that this conditioned space-time Brownian motion is the limit of a sequence of Markov processes constructed in section \ref{Markovchains}. We show in section \ref{Brownianinterval} how it is related to a Brownian motion  conditoned - in Doob's sense - to remain in an interval.

\section{The affine Lie algebra  $\hat{\mathfrak{sl_2}}(\C)$}\label{Basic notations and definitions}
We consider the affine Lie algebra $\hat{\mathfrak{sl}_2}(\C)$ associated to the generalized Cartan matrix $$A=\begin{pmatrix}
2&-2 \\
-2&2
\end{pmatrix}.$$ The reader is invited to refer to \cite{Kac} for a detailed description of this object. Let $\mathfrak{h}$ be a Cartan subalgebra of $\hat{\mathfrak{sl}_2}(\C)$. We denote by $S=\{\alpha_0,\alpha_1\}$ the set of simple roots and by   $\{\alpha_0^\vee,\alpha^\vee_1\}$ the set of simple coroots. Let $\Lambda_0$ be a fundamental weight such that $\langle \Lambda_0,\alpha^\vee_i\rangle=\delta_{i0}$, $i\in \{0,1\}$, and $\{\alpha_0,\alpha_1,\Lambda_0\}$ is a basis of $\mathfrak{h}^*$. We denote by $\mathfrak{h}_\R$ the subset of $\mathfrak{h}$ defined by 
$$\mathfrak{h}_\R=\{x\in \mathfrak{h}:  \langle\Lambda_0,x\rangle\in \R, \textrm{ and } \langle\alpha_i,x\rangle\in \R, i\in \{0,1\}\}.$$  Let $\delta=\alpha_0+\alpha_1$ be the so-called null root. We denote by $P$ (resp. $P_+$) the set of integral (resp. dominant) weights defined by 
$$P=\{\lambda\in \mathfrak{h}^*:  \langle \lambda,\alpha^\vee_i \rangle\in \Z, \, i=0,1\},$$
 $$(\textrm{resp. } P_+=\{\lambda\in P:  \langle\lambda,\alpha^\vee_i \rangle\ge 0, \, i=0,1\}).$$
The Cartan subalgebra  $\mathfrak{h}$ is equipped with a non degenerate symmetric bilinear form $( .,.)$ defined below, which  identifies $\mathfrak{h}$ and $\mathfrak{h}^*$, through the linear isomorphism 
\begin{align*}
\nu:\, \,&\mathfrak{h}\to \mathfrak{h}^*, \\
& h\mapsto (h,.).
\end{align*} We still denote by $(.,.)$ the induced non degenerate symmetric bilinear form   on $\mathfrak{h}^*$. It is defined on $\mathfrak{h}^*$ by  
 \begin{align*}
\left\{
    \begin{array}{ll}
       ( \Lambda_0,\alpha_1)&= 0 \\
        ( \Lambda_0,\Lambda_0)&= 0 \\
 (\delta,\alpha_1)&= 0 \\
( \Lambda_0,\delta)&=1\\
 (\alpha_1,\alpha_1)&=2.
    \end{array}
\right.
\end{align*}
The level of an integral weight $\lambda\in P$, is defined as the integer $(\delta,\lambda)$. For $k\in \N$, we denote by $P_k$ the set integral weights of level $k$. It is defined by
$$P_k=\{\lambda\in P: (\delta,\lambda)=k\}.$$
That is, an integral weight of level $k$  can be written 
$$k\Lambda_0+\frac{x}{2}\alpha_1+y\delta,$$
where $x\in \Z$, $y\in \C$, and a dominant weight of level $k$  can be written 
$$k\Lambda_0+\frac{x}{2}\alpha_1+y\delta,$$
where $x\in \{0,\dots,k\}$, $y\in \C$.
Recall the following important property : all weights of an highest weight irreducible representation of $\hat{\mathfrak{sl}_2}(\C)$ have the same level. 
\begin{notaetoile} For $\lambda\in \mathfrak{h}^*$, the projection of $\lambda$ on $vect\{\Lambda_0,\alpha_1\}$,   denoted $\bar{\lambda}$, is defined by   $\bar{\lambda}=x\Lambda_0+y\alpha_1$, when $\lambda=x\Lambda_0+y\alpha_1+z\delta$, $x,y,z\in \C$.
\end{notaetoile}

\paragraph{\bf Characters} For $\lambda\in P_{+}$, we denote by $\mbox{ch}_\lambda$ the character of the irreducible highest-weight module $V(\lambda)$  of $\hat{\mathfrak{sl}_2}(\C)$ with highest weight $\lambda$. That is
$$\mbox{ch}_\lambda(h) = \sum_{\mu \in P}\mbox{dim} (V(\lambda)_\mu)e^{\langle \mu ,h\rangle} , \quad h\in \mathfrak{h},$$
where $V(\lambda)_\mu$ is the weight space of $V(\lambda)$ corresponding to the weight $\mu$. The above series converges absolutely for every $h\in \mathfrak{h}$ such that $\mbox{Re}\langle \delta,h\rangle>0$ (see chapter $11$ of \cite{Kac}). For $\beta\in \mathfrak{h}^*$, we write $\mbox{ch}_\lambda(\beta)$ for $\mbox{ch}_\lambda(\nu^{-1}(\beta))$. We have
$$\mbox{ch}_\lambda(\beta) = \sum_{\mu \in P}\mbox{dim} (V(\lambda)_\mu)e^{( \mu ,\beta)} , \quad \beta\in \mathfrak{h}^*.$$ 
  The Weyl character's formula states that 
\begin{align*}
\mbox{ch}_\lambda(.)=\frac{\sum_{w\in W}\det(w)e^{(w(\lambda+\rho),.)}}{\sum_{w\in W}\det(w)e^{(w(\rho),.)}},
\end{align*}
where $\rho=2\Lambda_0+\frac{1}{2}\alpha_1$ and $W$ is the group of linear transformations  of $\mathfrak{h}^*$ generated by the reflections $s_{\alpha_0}$ and $s_{\alpha_1}$ defined by 
$$s_{\alpha_i}(x)=x-2\frac{(\alpha_i,x)}{(\alpha_i,\alpha_i)}\alpha_i, \, x\in \mathfrak{h}^*, \, i\in\{0,1\}.$$ 
As proved for instance in chapter $6$ of \cite{Kac},  the affine Weyl group W is the semi-direct product $T\ltimes W_0$ where $W_0$ is the Weyl group generated by $s_{\alpha_1}$ and $T$ is the group of transformations $t_{k}$, $k\in \Z$, defined by 
$$t_k(\lambda)=\lambda+k(\lambda,\delta)\alpha_1-(k(\lambda,\alpha_1)+k^2(\lambda,\delta))\delta, \,\, \lambda\in \mathfrak{h}^*.$$
Thus for $a\in \R^*$, $y\in \R_+^*$, and a dominant weight $\lambda$ of level $n\in \N^*$, such that $\lambda=n\Lambda_0+\frac{1}{2}x\alpha_1$, the Weyl character formula becomes 
\begin{align}\label{explicitchar1}
\mbox{ch}_\lambda(ia\alpha_1+y\Lambda_0)=\frac{\sum_{k\in \Z}\sin(a(x+1)+2ak(n+2))e^{-y(k(x+1)+k^2(n+2))}}{\sum_{k\in \Z}\sin(a+8ak)e^{-y(k+4k^2)}}.
\end{align}
Letting $a$ goes to zero in the previous identity, one also obtains that 
\begin{align}\label{explicitchar2}
\mbox{ch}_\lambda(y\Lambda_0)=\frac{\sum_{k\in \Z}(x+1+2k(n+2))e^{-y(k(x+1)+k^2(n+2))}}{\sum_{k\in \Z}(1+8k)e^{-y(k+4k^2)}},
\end{align}
for every $y\in \R_+^*$.
\section{Markov chains on the sets of integral or dominant weights}\label{Markovchains}
Let us choose for this section a dominant weight $\omega\in P_+$ and $h\in \mathfrak{h}^*_\R$ such that  $( \delta,h)>0$.
\paragraph{\bf Random walks on $P$}  We define a probability measure $\mu_\omega$ on $P$ letting 
\begin{align}\label{weightmeasure}
\mu_\omega(\beta)=\frac{\dim(V(\omega)_\beta)}{\mbox{ch}_\omega(h) }e^{\langle \beta ,h\rangle}, \quad \beta \in P.
\end{align}
If $(X(n),n \ge 0)$ is a random walk on $P$ whose increments are distributed according to $\mu_\omega$, it is important for our purpose to keep in mind that  the function 
$$x\in  \R\mapsto[\frac{\mbox{ch}_\omega(i\frac{x}{2}\alpha_1+h)}{\mbox{ch}_\omega(h)}]^n,$$
is the Fourier transform of the projection of $X(n)$ on $\R\alpha_1$.
\bigskip

\paragraph{\bf Markov chains on $P_+$} Let us consider for $\lambda\in P_+$ the following decomposition
$$\mbox{ch}_{\omega}\mbox{ch}_{\lambda}=\sum_{\beta\in P_+} m_{\lambda}(\beta)\mbox{ch}_{\beta},$$
where $m_{\lambda}(\beta)$ is the multiplicity of the module with highest weight $\beta$ in the decomposition of $V(\omega)\otimes V(\lambda)$, leads to the  definition a transition probability  $q_\omega$ on $P_+$ given by
\begin{align}\label{Markovdominant}
q_\omega(\lambda,\beta)=\frac{\mbox{ch}_\beta(h)}{\mbox{ch}_\lambda(h) \mbox{ch}_\omega(h)}m_{\lambda}(\beta),\quad \beta \in P_+.
\end{align}
Let us notice that if  $(\Lambda(n),n\ge 0)$ is a Markov process starting from $\lambda_0\in P_+$, with transition probabilities $q_\omega$ then
$$\E(\frac{\mbox{ch}_{\Lambda(n)}(ix+h)}{ch_{\Lambda(n)}(h)})=\frac{\mbox{ch}_{\lambda_0}(ix+h)}{\mbox{ch}_{\lambda_0}(h)}[\frac{\mbox{ch}_{\omega}(ix\alpha_1+h)}{ch_{\omega}(h)}]^n,$$
for every $x\in \R$. If $\lambda_1$ and $\lambda_2$ are two dominant weights  such that $\lambda_1= \lambda_2\,(mod\, \delta)$ then the irreducible modules $V(\lambda_1)$ and $V(\lambda_2)$ are isomorphic. Thus if we consider the random process $(\bar{\Lambda}(n),n\ge 0)$, where $\bar{\Lambda}(n)$ is the projection of  $\Lambda(n)$ on $vect\{\Lambda_0,\alpha_1\}$, then $(\bar{\Lambda}(n),n\ge 1)$ is a Markov process satisfying
\begin{align}\label{discrete}
\E(\frac{\mbox{ch}_{\bar\Lambda(n)}(ix\alpha_1+h)}{\mbox{ch}_{\bar\Lambda(n)}(h)})=\frac{\mbox{ch}_{\bar\lambda_0}(ix\alpha_1+h)}{\mbox{ch}_{\bar\lambda_0}(h)}[\frac{\mbox{ch}_{\omega}(ix\alpha_1+h)}{ch_{\omega}(h)}]^n,
\end{align}
for every $x\in \R$, where $\bar \lambda_0$ is the projection of $\lambda_0$  on $vect\{\Lambda_0,\alpha_1\}$.
More generally, for $n,m\in \N$, one gets
\begin{align}\label{discrete2}
\E(\frac{\mbox{ch}_{\bar\Lambda(n+m)}(ix\alpha_1+h)}{\mbox{ch}_{\bar\Lambda(n+m)}(h)}\vert \bar{\Lambda}(k),\,0\le k\le m)=\frac{\mbox{ch}_{\bar\Lambda(m)}(ix\alpha_1+h)}{\mbox{ch}_{\bar\Lambda(m)}(h)}[\frac{\mbox{ch}_{\omega}(ix\alpha_1+h)}{\mbox{ch}_{\omega}(h)}]^n,
\end{align}
for every $x\in \R$. Let us notice that if $\omega$ is a dominant weight of level $k$, and $\lambda_0$ a dominant weight of level $k_0$, then $\bar\Lambda(n)$ and $\Lambda(n)$ are  dominant weights of level  $nk+k_0$, for every $n\in \N$.

\section{A conditioned space-time Brownian motion}\label{Brownian}
\paragraph{\bf A class of space-time harmonic functions}
Considering the asymptotic of the previous characters, one obtains an interesting class of space-time harmonic functions involving the Jacobi's theta function $\theta$ defined by
$$\theta(z,\tau)=\sum_{n\in\Z} e^{\pi i n^2\tau+2\pi i n z},$$
for $z$ and $\tau$ two complex numbers, $\tau$ being in the upper half-plane. This is not surprising as the characters of affine Lie algebras  are themself a linear combination of theta functions (see \cite{Kac}).  For $a\in \R^*$, $x\in [0,t]$, if $(\lambda_n)_n$ is a sequence of dominant weights such that 
\begin{align*}
\lambda_n=[&nt]\Lambda_0+\frac{1}{2}(\lambda_n,\alpha_1)\alpha_1\\
\underset{n\to\infty}\sim& nt\Lambda_0 +nx\frac{1}{2}\alpha_1,\end{align*} 
then the sum
$$\underset{k\in\Z}\sum\sin(\frac{a}{n}((\lambda_n,\alpha_1)+1)+2\frac{a}{n}k([nt]+2))e^{-\frac{2}{n}(k((\lambda_n,\alpha_1)+1)+k^2([nt]+2))},$$
which is the numerator of $\mbox{ch}_{\lambda_n}(i\frac{a}{n}\alpha_1+\frac{2}{n}\Lambda_0)$ in the right-hand side of identity (\ref{explicitchar1}), converges, 
when $n$ goes to infinity, towards
$$\sum_{k\in \Z}\sin(ax+2kat)e^{-2(kx+k^2t)}.$$
\begin{defn} For $a\in \R^*$, we define a function $\phi_a$ on $\R\times \R^*_+$ letting 
$$\phi_a(x,t)=\frac{\pi}{\sh(a\pi)}\sum_{k\in \Z}\sin(ax+2kat)e^{-2(kx+k^2t)}, \, \,(x,t)\in \R\times \R_+^*.$$ 
\end{defn}
\noindent 
Similarly, considering the asymptotic of the numerator of $\mbox{ch}_{\lambda_n}(\frac{2}{n}\Lambda_0)$ in (\ref{explicitchar2}) leads naturally to the following definition.
\begin{defn}
We define a function $\phi_0$ on $\R\times\R^*_+$ letting 
$$\phi_0(x,t)=\sum_{k\in \Z}(x+2kt)e^{-2(kx+k^2t)}, \, \, (x,t)\in \R\times\R_+^*.$$
\end{defn}
\noindent
 Let us notice that  $\underset{a\to 0}\lim\phi_a=\phi_0$ and $$\phi_{\frac{n\pi}{t}}(x,t)=\frac{\pi}{\sh(\frac{n\pi^2}{t})}\sin(n\pi \frac{x}{t})\sum_{k\in \Z}e^{-2(kx+k^2t)},$$
for every $n\in\N^*$, $(x,t)\in   \R\times \R^*$.
\begin{prop}  For $a\in \R$, the function $$(x,t)\in \R\times \R_+^* \mapsto e^{\frac{a^2}{2}t}\phi_a(x,t),$$ is a space-time harmonic function, i.e. $\phi_a$ satisfies 
$$(\frac{1}{2}\frac{\partial^2}{\partial x^2}+\frac{\partial}{\partial t})\phi_a=-\frac{a^2}{2}\phi_a.$$ Moreover $\phi_a$ satisfies  the following boundary conditions
$$ \forall t\in \R_+^*, \,\, \left\{
    \begin{array}{ll}
      \phi_a(0,t)&=0 \\
    \phi_a(t,t)&=0.
    \end{array}
\right. $$
\end{prop}
\begin{proof} Actually, each summand of the sum in the definition of $e^{\frac{a^2}{2}t}\phi_a$ is a space-time harmonic function because for any $k\in \Z$, one has
$$e^{i(ax+2kat)-2(kx+k^2t)+\frac{a^2}{2}t}=e^{(ia-2k)x-\frac{1}{2}(ia-2k)^2t}.$$
The first boundary condition follows from the change of variable $k\mapsto -k$, whereas the last one follows from the change of variable $k\mapsto -1-k$.   
\end{proof}
\bigskip

\paragraph{\bf Some properties of the functions $\phi_a$, $a\in \R$}
\begin{lem}\label{Poisson}
$$\sum_{k\in \Z}\sin(ax+2kta)e^{-2(kx+k^2t)}=e^{\frac{x^2}{2t}-a^2\frac{t}{2}}\sum_{k\in \Z}\sqrt{\frac{\pi}{2t}}e^{-\frac{1}{2t}k^2\pi^2}\sh(k\pi a)\sin(k\frac{\pi}{t}x)$$
$$\sum_{k\in \Z}(x+2kt)e^{-2(kx+k^2t)}=e^{\frac{x^2}{2t}}\sum_{k\in \Z}\sqrt{\frac{\pi}{2t}}e^{-\frac{1}{2t}k^2\pi^2}k\pi \sin(k\frac{\pi}{t}x)$$
\end{lem}
\begin{proof}
As $\sin(ax+2kta)e^{-2(kx+k^2t)}=e^{\frac{x^2}{2t}}\sin(a(x+2kt))e^{-\frac{1}{2t}(x+2kt)^2}$, the first identity follows from a Poisson summation formula, which is obtained computing    the Fourier coefficients of the $2t$-periodic function $x\mapsto e^{-\frac{x^2}{2t}}\phi_a(x,t)$. The second identity follows similarly from the identity  $(x+2kt)e^{-2(kx+k^2t)}=e^{\frac{x^2}{2t}}(x+2kt)e^{-\frac{1}{2t}(x+2kt)^2}$. Let us notice that the second identity can  also be derived from the well known Jacobi's theta function identity
\begin{align}\label{Jacobitheta}
\frac{1}{\sqrt{\pi t}}\sum_{n\in \Z}e^{-\frac{1}{t}(n+x)^2}=\sum_{n\in \Z} \cos(2n\pi x)e^{-n^2\pi^2t},
\end{align}
which is valid for $x\in \R, \, t\in \R_+^*,$ and which is also a particular case of the Poisson summation formula  (see \cite{Bellman}).  Considering the partial derivative with respect to $x$  of the left and the right hand sides  in identity (\ref{Jacobitheta}) leads to the identity 
\begin{align}\label{Jacobithetader}
\frac{1}{\sqrt{\pi t}}\sum_{n\in \Z}\frac{1}{t}(n+x)e^{-\frac{1}{t}(n+x)^2}=\sum_{n\in \Z}n\pi \sin(2n\pi x)e^{-n^2\pi^2t},
\end{align}
for $x\in \R, \, t\in \R_+^*.$ As 
$$\phi_0(x,t)=\sum_{n\in \Z}2t(\frac{x}{2t}+n)e^{-2t(n+\frac{x}{2t})^2+\frac{x^2}{2t}},$$
one obtains the second replacing respectively $t$ and $x$  by $\frac{1}{2t}$ and $\frac{x}{2t}$ in (\ref{Jacobithetader}).
\end{proof}  

\begin{lem} Let $t\in \R_+^*$, and $x\in]0,t[$. If   $(\lambda_n)_n$ is a sequence of dominant weights such that 
$$\lambda_n\sim nt\Lambda_0 +nx\frac{1}{2}\alpha,$$ then 
$$\underset{n\to \infty}\lim\frac{\mbox{ch}_{\lambda_n}(\frac{ia}{n}\alpha_1+\frac{2}{n}\Lambda_0)}{\mbox{ch}_{\lambda_n}(\frac{2}{n}\Lambda_0)}=\frac{\phi_a(x,t)}{\phi_0(x,t)}.$$
\end{lem}
\begin{proof} Lemma \ref{Poisson}   implies   that
$$\underset{n\to\infty}\lim\frac{\underset{k\in \Z}\sum\sin(\frac{a}{n}(1+8k))e^{-\frac{2}{n}(k+4k^2)}}{\frac{1}{n}\underset{k\in \Z}\sum(1+8k)e^{-\frac{2}{n}(k+4k^2)}}=\frac{\sh(a\pi)}{\pi}.$$
Thus the lemma follows from identities (\ref{explicitchar1}) and (\ref{explicitchar2}).
\end{proof}
\begin{prop} \label{propphi} Let $a\in \R^*$, and $t\in \R_+^*$. Then
\begin{enumerate} 
\item The function $\phi_0(.,t)$ is $C^\infty$ on $[0,t]$,
\item the function $\frac{\phi_a(.,t)}{\phi_0(.,t)}$ is bounded on $[0,t]$,
\item $\forall x\in]0,t[$,   $\phi_0(x,t)\ne 0$,
\item the function $\phi_0(.,t)$ doesn't change of sign on $[0,t]$.
\end{enumerate}
\end{prop}
\begin{proof} The first property   follows immediately from a dominated convergence theorem. As for any $r\in\R$ and $y\in \R_+^*$,  $\frac{\mbox{ch}_\lambda(ir\alpha_1+y\Lambda_0)}{\mbox{ch}_\lambda(y\Lambda_0)}$ is a Fourier transform of a probability measure, it is bounded by $1$. Previous lemma implies that 
$$\forall x\in ]0,t[,\quad \vert\frac{\phi_a(x,t)}{\phi_0(x,t)}\vert \le 1.$$ As the function $\frac{\phi_a(.,t)}{\phi_0(.,t)}$ is easily shown to be continuous on $[0,t]$, the second property follows. For the third property, we notice that the function $\phi_{\frac{\pi}{t}}$ is defined by
$$\phi_{\frac{\pi}{t}}(x,t)=2t\sin(\frac{x}{t}\pi)\sum_{k\in \Z}e^{-2(kx+k^2t)}, \,\, (x,t)\in \R\times\R_+^*.$$ Thus   for $x\in [0,t]$, 
$$\phi_{\frac{\pi}{t}}(x,t)=0\Leftrightarrow  x\in\{0,t\}.$$
\noindent
The function $\frac{\phi_{\frac{\pi}{t}}(.,t)}{\phi_0(.,t)}$ being bounded on $[0,t]$, the third property follows. The fourth one is an immediate consequence of the first and the third ones.
\end{proof} 
Let us enounce a very classical result on Fourier series that will be used   to prove  proposition \ref{characterization}
 \begin{lem}\label{Fourier} Let $t$ be a positive real number and
  $f:[0,t]\to \R$ be a function such  that $f(0)=f(t)=0$, which is $C^3$ on $[0,t]$. Then  letting $c_n=\frac{1}{t}\int_0^t\sin(\frac{z}{t}n\pi)f(z)\,dz$, $n\in \N$, 
the series $\sum_nnc_n$ converges absolutely and
$$f(x)=\sum_{n=1}^{+\infty}c_n \sin(\frac{x}{t}n\pi),$$  for every $x\in [0,t]$.
\end{lem}
\begin{prop} \label{characterization} Let $t$ be a positive real number and  $\mu$ be a probability measure on $[0,t]$.  
 Then  $\mu$ is characterized by the quantities
$$\int_0^t\frac{\phi_{n\pi/t}(x,t)}{\phi_0(x,t)}\,\mu(dx), \,n\in \N.$$
\end{prop}
\begin{proof}   For $t\in \R_+^*$, $x\in \R$, we let 
$$e(x,t)=\underset{k\in\Z}\sum e^{-2(kx+k^2t)}.$$ 
Let $u$ be a   $C^3$ function on  $[0,t]$. We first notice that the function $\frac{u(.)\phi_0(.,t)}{e(.,t)}$ satisfies the condition of lemma \ref{Fourier}. We let for $n\in \N^*$, 
$$c_n=\frac{1}{t}\int_0^t \frac{u(x)\phi_0(x,t)}{e(x,t)}\sin(\frac{x}{t}n\pi)\, dx.$$ One has
\begin{align*}
\int_0^tu(x)\, \mu(dx)&=\int_0^tu(x)\frac{\phi_0(x,t)}{e(x,t)}\frac{e(x,t)}{\phi_0(x,t)}\, \mu(dx)\\
&=\int_0^t\sum_{n=1}^{+\infty}c_n\sin(\frac{x}{t}n\pi)\frac{e(x,t)}{\phi_0(x,t)}\, \mu(dx)
\end{align*}
Using the two identities of lemma \ref{Poisson} one obtains
\begin{align*}
\int_0^tu(x)\, \mu(dx)&=\int_0^t\sum_{n=1}^{+\infty}c_n\frac{\sum_{k\in \Z }e^{-\frac{\pi^2}{2t}k^2}\sin((n+k)\pi\frac{x}{t})}{\sum_{k\in \Z}e^{-\frac{\pi^2}{2t}k^2}k\pi\pi\sin(k\pi\frac{x}{t})}\,\mu(dx)\\
&=\sum_{n=1}^{+\infty}c_n\int_0^t\frac{\sum_{k\in \Z}e^{-\frac{\pi^2}{2t}k^2}\sin((n+k)\pi\frac{x}{t})}{\sum_{k\in \Z }e^{-\frac{\pi^2}{2t}k^2}k\pi\pi\sin(k\pi\frac{x}{t})}\,\mu(dx)
\end{align*}
where the last identity follows from the fact that the series $\sum nc_n$ is absolutely convergent.  Thus 
\begin{align*}
\int_0^tu(x)\, \mu(dx)&=\sum_{n=1}^{+\infty}\frac{c_n}{2}\frac{\sh(n\frac{\pi^2}{t})}{\pi}\int_0^t\frac{\phi_{n\pi/t}(x,t)}{\phi_0(x,t)}\,\mu(dx).
\end{align*}
\end{proof}
\bigskip

\paragraph{{\bf A conditioned space-time Brownian motion}} 
Let us denote by $C$ the fundamental Weyl chamber defined by
$$C=\{x\in \mathfrak{h}^* : (x,\alpha_i)\ge 0, i\in\{0,1\}\}.$$ 
That is, an element of $C$ can be written 
$$t\Lambda_0+\frac{x}{2}\alpha_1+y\delta,$$
where $t\in \R_+, x\in[0,t], y\in \C.$
 For $x\in \R$, we denote by $\mathbb{W}_x$ the Wiener measure on the set $C(\R_+)$ of real valued continuous functions on $\R_+$, under which the coordinate process $(X_t,t\ge 0)$ is a Brownian motion starting from $x$, and denote the natural filtration of $(X_t)_{t\ge 0}$ by $(\mathcal{F}_t)_{t\ge 0}$. One considers the stopping times $T_u$, $u\in \R_+$,  defined by $$T_u=\inf\{t\ge 0: X_t=0 \textrm{ or } X_t=t+u\}.$$ 
Proposition  \ref{propphi} ensures that $\phi_0(X_s,s+u)$ doesn't change of sign whenever $s\in[0,T_u]$. Let $x$ be a positive real number such that $u>x$. One has $\mathbb{W}_x(T_u>0)=1$. The function $\phi_0$ being space-time harmonic, the process $\phi_0(X_t,t+u)$ is a local martingale. Actually, each summand of the sum is a local martingale, for which the quadratic variation is easily shown to be integrable, so that, each summand is a true Martingale. As their sum converges absolutely in $L_2$ norm, one obtains that $\phi_0(X_t,t+u),t\ge 0,$ is a true martingale. As $\phi(X_{T_u},T_u+u)=0$, one defines a measure $\mathbb{Q}_{x,u}$ on $C(\R_+)$ letting 
$$\mathbb{Q}_{x,u}(A)=\E_x(\frac{\phi_0(X_t,t+u)}{\phi_0(x,u)}1_{\{T_u>t\}\cap A}), \, A\in \mathcal{F}_t.$$ 
Let $r$ and $s$ be two positive real numbers  such that $r<t$. Using that  $$T_u=T_{u+r}\circ\theta_r+r\, \textrm{ on } \{T_u\ge r\},$$ where $\theta_r$ the shift operator defined by
$$\forall t\in \R_+^*, \, X_t\circ \theta_r=X_{t+r},$$ one easily proves that $(X_t,t\ge 0)$ is an inhomogeneous Markov process under $\mathbb{Q}_{x,u}$ satisfying 
\begin{align}\label{Markovkernel}
\E_{\mathbb{Q}_{x,u}}(f(X_{t+r})\vert \mathcal{F}_r)&=\E_{X_r}(\frac{\phi_0(X_t,t+r+u)}{\phi_0(X_r,r+u)}f(X_t)1_{T_{u+r}>t})\nonumber\\
&=\E_{\mathbb{Q}_{X_r,r+u}}(f(X_t)),
\end{align}
for any real valued measurable bounded function $f$.
\begin{prop}\label{phiQ} For $r,t,u\in \R_+^*$, $x\in ]0,u[$, and $a\in \R$, one has 
\begin{align}\label{continuous}
\mathbb{Q}_{x,u}(\frac{\phi_a(X_t,t+u)}{\phi_0(X_t,t+u)})=\frac{\phi_a(x,u)}{\phi_0(x,u)}e^{-\frac{a^2}{2}t}.
\end{align}
and
\begin{align}\label{continuouscond}
E_{\mathbb{Q}_{x,u}}(\frac{\phi_a(X_{t+r},t+r+u)}{\phi_0(X_{t+r},t+r+u)}\vert \mathcal{F}_r)=\frac{\phi_a(X_r,r+u)}{\phi_0(X_r,r+u)}e^{-\frac{a^2}{2}t}.
\end{align}
\end{prop}
\begin{proof}
 One proves as previously that $$(e^{\frac{a^2}{2}t}\phi_a(X_t,t+u), t\ge 0)$$ is a true martingale, which implies that
$$\mathbb{W}_x(\phi_a(X_t,t+u)e^{\frac{a^2}{2}t}1_{T_u<t})=0,$$
and  identity (\ref{continuous}). The second identity follows, using (\ref{Markovkernel}).
\end{proof}

\section{The conditioned Brownian motion and the Markov chains on the set of dominant weights}\label{convergence}
Let us focus on the Markov chains defined in section \ref{Markovchains} when  $\omega=\Lambda_0$. We recall that the weights occurring in $V(\Lambda_0)$ are
$$\Lambda_0 + k\alpha_1-(k^2+s)\delta, \, k\in \Z, \, s\in \N,$$
with respective multiplicities $p(s)$, the number of partitions of $s$ (see for instance chapter $9$ in \cite{HongKang}). If we consider, for $h=\frac{1}{2}(h_1\alpha_1+h_2\Lambda_0)$, with $h_1\in\R,h_2\in \R_+^*$, the associated probability measure $\mu_{\Lambda_0}$ defined  by (\ref{weightmeasure})   and the associated random walk $(X(n),n\ge 0)$, then its projection  on $\Z\alpha_1$ is a random walk with increments distributed according to a probability measure $\bar{\bar{\mu}}_{\Lambda_0}$ defined by
$$\bar{\bar\mu}_{\Lambda_0}(k)=C_he^{kh_1-\frac{h_2}{2}k^2}, \, k\in \Z,$$
where $C_h$ is a normalizing constant depending on $h$.
\bigskip
\paragraph{\bf The main theorem}
For $n\in \N^*$, we consider a random walk $(X^{n}_k,k\ge 0)$ starting from $0$, whose increments are distributed according to probability measure  $\mu_{\Lambda_0}$ associated to $h=\frac{2}{n}\Lambda_0$. If we denote by $(\bar{\bar{X}}^{n}_k,k\ge 0)$ its projection on $\Z\alpha_1$, standard method shows that the sequence of processes $(\frac{2}{n}\bar{\bar{X}}^{n}_{[nt]},t\ge0)$ converges towards a standard Brownian motion on $\R$ when $n$ goes to infinity.

\noindent
Let $x$ and $u$ be two positive numbers such that $x<u$. For $n\in \N^*$, we consider a Markov process $(\Lambda^{n}_k,k\ge 0)$ starting from $[nu]\Lambda_0+[xn]\frac{1}{2}\alpha_1$, with the transition probability $q_\omega$ defined by (\ref{Markovdominant}), with $\omega=\Lambda_0$ and $h=\frac{2}{n}\Lambda_0$. It is important to notice that $(\Lambda^{n}_k,\delta)=[nu]+k$ for every $k\in \N$. If $\bar\Lambda_k^{n}$ is the projection of $\Lambda_k^{n}$ on $vect\{\Lambda_0,\alpha_1\}$ for every $k\in \N$ and $n\in \N^*$, then the following convergence holds. 

\begin{theo}
The sequence of processes $(\frac{1}{n}\bar\Lambda^{n}_{[nt]}, t\ge 0)$ converges when $n$ goes to infinity towards the process  $((t+u)\Lambda_0+\frac{X_t}{2}\alpha_1,t\ge 0)$ under $\mathbb{Q}_{x,u}$ 
\end{theo}
\begin{proof}
Let  $t\in \R_+^*$. We denote  by $\mu^{n}_{t}$ the law of  $\frac{1}{n}(\Lambda^{n}_{[nt]},\alpha_1)$,   for $n\in \N$. The probability measure $\mu_t^n$   is carried by $[0,t+u]$. The intervale $[0,t+u]$ being a compact set, the space of probability measures on $[0,t+u]$ endowed with the weak topology is also compact.  Suppose that a subsequence of $(\mu^{n}_{t})_{n}$ converges towards $\mu_t$. 
For $\lambda\in P_m^+$, one has
$$\frac{\mbox{ch}_\lambda(\frac{a}{n}\alpha_1+\frac{2}{n}\Lambda_0)}{\mbox{ch}_\lambda(\frac{2}{n}\Lambda_0)}=\frac{\phi_a(\frac{1}{n}((\lambda,\alpha_1)+1),\frac{1}{n}(m+2))}{\phi_0(\frac{1}{n}((\lambda,\alpha_1)+1),\frac{1}{n}(m+2))}\frac{\phi_0(\frac{1}{n},\frac{4}{n})}{\phi_a(\frac{1}{n},\frac{4}{n})},$$
for any $a\in \R$, and $n\in \N^*$. 
The function  $(x,t)\mapsto\frac{\phi_a(x,t+u)}{\phi_0(x,t)}$ can be shown to be uniformly continuous on $\{(x,t)\in \R\times[0,T]: 0\le x\le u+t\}$ for every $T\in\R_+$. As $\underset{n\to\infty}\lim \frac{\phi_0(\frac{1}{n},\frac{4}{n})}{\phi_a(\frac{1}{n},\frac{4}{n})}=1$, and 
$$\Big[\frac{\mbox{ch}_{\Lambda_0}(\frac{a}{n}\alpha_1+\frac{2}{n}\Lambda_0)}{\mbox{ch}_{\Lambda_0}(\frac{2}{n}\Lambda_0)}\Big]^{[nt]}=\mathbb{E}(e^{ia\frac{2}{n}\bar{\bar X}_{[nt]}^{(n)}}),$$ 
identity (\ref{discrete}) implies that $\mu_t$ satisfies
$$\int_0^{t+u}\frac{\phi_a(z,t+u)}{\phi_0(z,t+u)}\, \mu_t(dz)=\frac{\phi_a(x,u)}{\phi_0(x,u)}e^{-\frac{a^2}{2}t}.$$ 
Proposition \ref{characterization}  implies that $(\mu^{n}_{t})_{n}$ converges towards $\mu_t$ and proposition \ref{phiQ} implies that $\mu_t$ is the distribution of $X_t$ under $\Q_{x,u}$. Convergence of the sequence of random processes $(\frac{1}{n}(\Lambda^n_{[nt]},\alpha_1), t\ge 0)$ - in the sense of finite dimensional distributions convergence - follows similarly from identity (\ref{discrete2}) and (\ref{continuouscond}).  

\end{proof}

\section{Brownian motion conditioned to remain in an interval}\label{Brownianinterval}
In this section we discuss the connection between the conditioned Brownian motion constructed in this paper and the  Brownian motion conditioned - in the sense of Doob - to remain in an interval. The connection is not surprising when we keep in mind that the dominant term in a character of a highest weight irreducible module of  an affine algebra involves the so-called asymptotic dimensions, which are related to eigenfunctions for the Laplacian on an interval (see  chapter 13 of \cite{Kac}). Let $u\in \R_+^*$. The function $h$ defined on $[0,u]$ by $$h(x)=\sin(\pi\frac{x}{u}), \, x\in [0,u],$$ is the Dirichlet eigenfunction on the interval $[0,u]$  corresponding to the eigenvalue $-\frac{\pi^2}{u^2}$ at the bottom of the spectrum.  Brownian motion conditioned - in the sense of Doob - to remain in the interval $[0,u]$, has the Doob-transformed semi-group $(q_t)_{t\ge 0}$ defined for $t\in \R_+^*$ by
$$q_t(x,y)=\frac{h(y)}{h(x)}e^{\frac{\pi^2}{u^2} \frac{t}{2}}p_t^0(x,y), \quad x,y\in ]0,u[, $$
where $p_t^0$ is the semi-group of the standard Brownian motion on $\R$, killed on the boundary of $[0,u]$.

For $c\in ]0,1[$, one defines  a space-time harmonic function  $\phi_0^{(c)}$ on $\R\times\R_+^*$   letting 
$$ \phi_0^{(c)}(x,t)=\phi_0(cx,c^2 t),$$
for $x,t\in \R\times\R_+^*$. This function satisfies the following boundary conditions 
$$ \forall t\in \R_+^*, \,\, \left\{
    \begin{array}{ll}
      \phi_0^{(c)}(0,t)&=0 \\
  \phi_0^{(c)}(ct,t)&=0.
    \end{array}
\right. $$ 
As in section \ref{Brownian}, one defines for a real number $x$ satisfying $0<x<u$,   a probability $\mathbb Q_{x,u}^{(c)}$ on $C(\R_+)$  letting 
$$\mathbb{Q}^{(c)}_{x,u}(A)=\E_x(\frac{\phi^{(c)}_0(X_t,t+\frac{u}{c})}{\phi^{(c)}_0(x,\frac{u}{c})}1_{\{T^{(c)}_u>t\}\cap A}), \, A\in \mathcal{F}_t,$$
where $T_u^{(c)}=\inf\{s\ge 0: X_s=0 \textrm{ or } X_s=cs+u\}.$ Thus, under the probability measure $\mathbb{Q}_{x,u}^{(c)}$, $(t+\frac{u}{c},X_t)_{t\ge 0}$ is a space-time Brownian motion starting from $(\frac{u}{c},x)$, conditioned to remain in  the domain
$$\{(r,z)\in \R\times \R: 0\le z\le cr\}.$$
\begin{theo} The probability measure $Q_{x,u}^{(c)}$ converges, when $c$ goes to $0$, towards the law of a standard Brownian starting from $x$, conditioned - in the sense of Doob - to remain in $[0,u]$.
\end{theo}
\begin{proof} Lemma \ref{Poisson}  easily implies that
$$\underset{c\to 0}\lim \frac{\phi_0^{(c)}(y,t+\frac{u}{c})}{\phi_0^{(c)}(x,\frac{u}{c})}=\frac{\sin(y\frac{\pi}{u})}{\sin(x\frac{\pi}{u})}e^{\frac{\pi^2}{u^2} \frac{t}{2}},$$
for every $y\in[0,u]$, $t>0$, which implies the theorem, as the quotient inside the  limit is uniformly bounded for $y\in [0,ct+u]$ and $c\in ]0,1[$.
\end{proof}


\begin{thebibliography}{00}
\bibitem{Bellman}{{\sc R. Bellman,} {\it A brief introduction to theta functions}, Holt, Rinehart and Winston, 1961.}
    \bibitem{BBO}{{\sc Ph.\ Biane, Ph.\ Bougerol and N. O'Connell}, {\it Littelmann paths and Brownian paths,} Duke Math. J., 130, no. 1, 127-167, 2005.}
    \bibitem{Crank}{{\sc Crank J.,} {\it Free and Moving Boundary Problems}, Clarendon Press, Oxford, 1984.}
    \bibitem{Doob}{{\sc Doob, J. L.,} {\it Classical Potential Theory and Its Probabilistic Counterpart,} Springer, New York.1984.}
\bibitem{Kac}{{\sc V. G. Kac,} {\it Infinite dimensional Lie algebras,} third edition, Cambridge university press, 1990.}
\bibitem{LLP}{{\sc C. Lecouvey, E. Lesigne, M. Peign\'e}, {\it Conditioned random walks from Kac-Moody root systems}, 	arXiv:1306.3082 [math.CO], 2013.}
\bibitem{Littelmann}{{\sc P. Littelmann}, {\it Paths and root operators in representation theory}, Annals of Mathematics 142, pp. 499--525, 1995.}\bibitem{HongKang}{{\sc J. Hong and S.-J. Kang,} {\it Introduction to Quantum groups and Crystal Bases,} American mathematical society, 2002.}
\end{thebibliography}
\end{document}